\documentclass{birkjour}
\usepackage{fourier}
\usepackage{ulem}
\usepackage{amsthm}
\usepackage{mathrsfs}
\usepackage{eufrak}
\usepackage{graphicx}
\usepackage{stmaryrd}
\usepackage{float}
\usepackage{amsmath}
\usepackage[all]{xy}
\usepackage{xypic}
\usepackage{hyperref}
\usepackage{tikz}
\usepackage{breqn}
\usepackage{lineno}

\usepackage{multicol}
\usepackage{multirow}
\usepackage{array}
\usepackage{booktabs}
\usepackage{rotating}

\hypersetup{
  pdftitle   = {},
  pdfauthor  = {},
  pdfcreator = {\LaTeX\ with package \flqq hyperref\frqq}
}

\DeclareMathAlphabet{\mathpzc}{OT1}{pzc}{m}{it}

\newtheorem{theorem}{Theorem}[section]
\newtheorem*{theorem*}{Theorem}
\newtheorem{proposition}[theorem]{Proposition}

\newtheorem*{lemma*}{Lemma}
\newtheorem{corollary}[theorem]{Corollary}
\newtheorem{conjecture}[theorem]{Conjecture}
\newtheorem*{conjecture*}{Conjecture}

\theoremstyle{definition}
\newtheorem{definition}[theorem]{Definition}
\newtheorem{example}[theorem]{Example}

\theoremstyle{remark}

\DeclareMathOperator{\im}{im}

\def\Tiny{\fontsize{2pt}{2pt}\selectfont}

\newcommand{\X}{X+\xi}

\newcommand{\J}{\mathcal{J}}
\newcommand{\TM}{TM \oplus T^*M}

\numberwithin{equation}{section}

\begin{document}

\title[Mirror symmetry of elliptic curves and generalized complex geometry]{A remark about mirror symmetry of elliptic curves and generalized complex geometry}

\author{Leonardo Soriani Alves and Lino Grama}
\address{Department of Mathematics - IMECC, University of Campinas - Brazil}
\email{leo.soriani@gmail.com, linograma@gmail.com}
\thanks{LSA is supported by Fapesp grant no. 2013/04034-7. LG is supported by Fapesp grant no. 2014/17337-0 and CNPq grant 476024/2012-9.}
\begin{abstract}
In this short note we describe the isomorphism of generalized complex structure between $T$-dual manifolds introduced by Cavalcanti-Gualtieri, in the case of elliptic curves. We also compare this isomorphism with the mirror map for elliptic curves described by Polishchuk and Zaslow. 
\end{abstract}

\maketitle

\section{Introduction}
Mirror symmetry predicts that in the mirror manifolds $M$ and $M^{\vee}$, the symplectic geometry of $M$ can be related to the complex geometry of $M^{\vee}$ and vice-versa. This correspondence between mirror manifolds has a geometric description by Strominger, Yau and Zaslow via $T$-duality \cite{SYZ}. 

Kontsevich provides another interpretation for mirror symmetry, also known as homological mirror symmetry(HMS): two Calabi-Yau manifolds $M$ and $M^{\vee}$ are mirror if the derived category of coherent sheaves on $M$ is equivalent to the derived Fukaya category of $M^{\vee}$.

In the case of elliptic curves, the HMS conjecture was proved by Polishchuk and Zaslow \cite{PZ}.

Cavalcanti and Gualtieri in \cite{CG} describe how to transport generalized complex structures between $T$-dual manifolds (in the sense of Bouwknegt, Hannabuss and Mathai) using an isomorphism of sections of the vector bundle where are defined the generalized complex structures.


In this work we explicitly describe the Cavalcanti-Gualtieri isomorphism in the case of elliptic curves, and we remark that in the case of trivial $B$-fields this isomorphism coincides with the mirror map described by Polishchuk-Zaslow as well SYZ-mirror symmetry. 

In the first three sections we briefly review generalized complex geometry, $T$-duality in the sense of BHM and mirror symmetry for elliptic curves. In the last section we construct the Cavalvanti-Gualtieri map for elliptic curves and we compare this with the other versions of mirror symmetry.



\section{Generalized complex geometry}

We will recall the basics of generalized complex geometry. A detailed description of this theory can be found in \cite{Gua}. Let $M$ be an $n$-dimensional smooth manifold and $H\in \varOmega^3(M)$ a closed $3$-form. The sum of the tangent and cotangent bundles $\TM$ has a natural symmetric bilinear form of signature $(n,n)$ defined by $$\langle X+\xi,Y+\eta \rangle=\frac{1}{2}(\eta(X)+\xi(Y))$$ and a bracket of sections called the Courant bracket: $$[\X,Y+\eta]_H=[X,Y]+\mathcal{L}_X\eta-i_Yd\xi+i_Xi_YH.$$

If $B\in \varOmega^2(M)$, we can see it as a map $TM \to T^*M$ given by $X \mapsto B(X)=i_XB$. Then we define the orthogonal automorphism $exp(B): \TM \to \TM$ $$exp(B):=\left(\begin{array}{cc}
1 & 0\\
B & 1\end{array}\right).$$

\begin{definition}
A generalized complex structure on $(M^{2n},H)$ is an orthogonal automorphism $\J:\TM \to \TM$ satisfying $\J^2=-1$ whose $i$-eigenbundle is involutive under the Courant bracket.
\end{definition}

The remarkable feature of the generalized complex structures is that they encompass complex and symplectic structures as special cases. If $J$ and $\omega$ are, respectively, complex and symplectic structures on $M$, then $$\J_J:=\left(\begin{array}{cc}
 -J & 0\\
  0 & J^* \end{array}\right)\ \mbox{and}\  \J_{\omega}:= \left(\begin{array}{cc}
0 & -\omega^{-1}\\
\omega & 0\end{array}\right)$$ are generalized complex structures on $M$.

Given a generalized complex structure $\J$, we can use a $2$-form $B\in \varOmega^2(M)$ to obtain $exp(-B)\circ \J \circ exp(B)$, which is again a generalized complex structure thanks to the orthogonality of $exp(B)$.

\begin{definition}\label{Bfield}
If $\J$ is a generalized complex structure on $M$ and $B\in \varOmega^2(M)$, the generalized complex structure $exp(-B)\circ \J \circ exp(B)$ is called $B$-field transform of $\J$. If $\J$ comes from a symplectic (complex) structure, $exp(-B)\circ \J \circ exp(B)$ is called $B$-symplectic ($B$-complex) structure.
\end{definition}

\section{T-duality in the sense of Bouwknegt, Hannabuss and Mathai}\label{s1}

We start with the definition of T-duality for principal torus bundles provided with a closed $3$-form, following \cite{BHM}.
\begin{definition}
Let $M$ and  $\tilde{M}$ be principal $T^k$-bundles over the same base $B$ and let $H \in \varOmega^3(M)$, $\tilde{H} \in \varOmega^3(\tilde{M})$ be invariant closed $3$-forms. Let $M \times_B \tilde{M}$ be the fiber product and consider the diagram 


$$\xymatrix{
    &(M\times_B \tilde{M},p^*H-\tilde{p}^*\tilde{H}) \ar[ld]^{}  \ar[rd]_{}& \\
(M,H)\ar[dr]&  & (\tilde{M},\tilde{H})\ar[ld]\\
               & B  & }$$

We say that $(M,H)$ and $(\tilde{M},\tilde{H})$ are T-dual if $p^*H-\tilde{p}^*\tilde{H}=dF$, where $F \in \varOmega^2(M \times_B \tilde{M})$ is a $T^{2k}$-invariant non-degenerate $2$-form on the fibers . The product $M\times_B \tilde{M}$ is called the correspondence space.
\end{definition}





Given T-dual spaces $M$ and $\tilde{M}$, Cavalcanti and Gualtieri \cite{CG} constructed an isomorphism $\varphi$ between the spaces of invariant sections of $TM\oplus T^*M$ and $T\tilde{M}\oplus T^*\tilde{M}$ preserving the bilinear form and the twisted Courant bracket. Thus one can transport $T^k$-invariant generalized complex structures between $M$ and $\tilde{M}$.

Essentially, $\varphi$ is the composition of the pull back to the correspondence space, $B$-transform defined by $-F$ on the correspondence space and push-forward to $\tilde{M}$:
\begin{equation}\label{fi}
\varphi(\X)= \tilde{p}_*(\hat{X})+p^*\xi-F(\hat{X}),
\end{equation}
where $\hat{X}$ is the only lift of $X$ to $M \times \tilde{M}$ such that $p^*\xi-F(\hat{X})$ is a basic form. The existence and uniqueness of such a lift is guaranteed by the non-degeneracy of $F$.



\begin{theorem}\label{fiso}
The map $\varphi$ is an isomorphism; that is, for all $u,v \in (\TM)/T^k$
$$\langle \varphi(u),\varphi(v) \rangle=\langle u,v \rangle\ \ \mbox{ and }\ \ [\varphi(u),\varphi(v)]_{\tilde{H}}= \varphi([u,v]_H).$$
\end{theorem}


Now we use $\varphi$ to transport invariant generalized complex structures:

\begin{corollary}\label{coro}
Let $(M,H)$ and $(\tilde{M},\tilde{H})$ be T-dual spaces. If $\J$ is a invariant generalized complex structure on $M$, then $$\tilde{\J}:=\varphi^{-1}\circ \J \circ \varphi$$ is an invariant generalized complex structure on $\tilde{M}$.
\end{corollary}
%

\section{Mirror of elliptic curves}

In this section we briefly review the well-known results about mirror symmetry for elliptic curves: homological mirror symmetry and T-duality. Classical references are \cite{Aur}, \cite{Kon} and \cite{PZ}. We start by recalling the complex and symplectic structure on elliptic curves.



\subsection{Complex and symplectic structures on elliptic curves}

\begin{definition}
An elliptic curve is a compact Riemann surface of genus one. 
\end{definition}

Topologically, every elliptic curve is homeomorphic to a  2-dimensional torus. 

\begin{definition}
Let $V$ be a vector space and $\varLambda$  a subgroup of $V$. We say that  $\varLambda$ is a \textit{lattice} if it is discrete and the quotient $V/\varLambda$ is compact.
\end{definition}

\begin{example}
Let $V$ be a vector space with base  $\{v_1,\dots,v_n\}$. Then $\varLambda=v_1\mathbb{Z}\oplus \dots \oplus v_n\mathbb{Z}$ is a lattice on $V$.
\end{example}

Every elliptic curve can be described by the quotient $\mathbb{C}/\varLambda$ where $\varLambda$ is a lattice on $\mathbb{C}$. Since an elliptic curve is a quotient of $\mathbb{C}$, it inherits a complex structure and a group structure.




\begin{proposition}\label{moduli-curv}
Two complex tori  $(\mathbb{C}/\varLambda_1)$ and $(\mathbb{C}/\varLambda_2)$  are isomorphic if, and only if, there exist $c \in \mathbb{C}^*$ such that  $\varLambda_1=c\varLambda_2$.
\end{proposition}

Given a complex number $\tau$ with $\im \tau >0$, we construct the elliptic curve $E=\mathbb{C}/(\mathbb{Z}\oplus \tau\mathbb{Z})$, and the complex structure on $E$ is determined by $\tau$. Of course different complex numbers give rise to bi-holomorphic elliptic curves. For more details about the moduli space of elliptic curves, see for instance \cite{Hain}.


In the mirror symmetry setting we consider complexified symplectic structures $\omega_{\mathbb{C}}=B+i\omega$ on the elliptic curve $E$, where $\omega$ is a symplectic form and $B$ is a closed 2-form on $E$ called {\it B-field}.  As $H^2(E,\mathbb{R})=\mathbb{R}$ a complexified symplectic form is again determined by a complex number $\rho=b+ia$ with $a>0$ (since $a$ corresponds to the area of $E$) where $a$ determines a symplectic struture on $E$ and $b$ determines a $2$-form. In the language of generalized complex structures it corresponds to a $B$-symplectic structure (see Definition \ref{Bfield})with $$\omega=\left(\begin{array}{cc}
0 & a\\
-a & 0\end{array}\right)\ \ \mbox{and}\ \ B=\left(\begin{array}{cc}
0 & b\\
-b & 0\end{array}\right).$$



\subsection{SYZ mirror symmetry}
Mirror symmetry is a phenomenon predicted  by string theory, where questions about symplectic geometry of a manifold 
$M$ can be translated into the complex geometry of a manifold $M^{\vee}$, and vice-versa.  The manifold $M^{\vee}$ is called the mirror of $M$.

Strominger-Yau-Zaslow \cite{SYZ} suggest a geometric interpretation to mirror symmetry: the SYZ mirror map in the case where there are no singularities is, by definition, $T$-duality. 

We will briefly review the SYZ-construction.


\begin{definition}
A Kähler manifold $(M,\omega,J)$ of complex dimension $n$ is called Calabi-Yau if its canonical bundle is holomorphically trivial; that is, there is a globally defined holomorphic volume form $\varOmega \in \varOmega^{n,0}(M)$.
\end{definition}

One can write the restriction of $\varOmega$ to a Lagrangian submanifold $X\subset M$ as $$\varOmega|_L=\psi vol_g,$$ where $\psi\in C^{\infty}(L,\mathbb{C}^*)$ and $vol_g$ is the volume form  induced by the Kähler metric $g:=\omega(\cdot,J\cdot)$.

\begin{definition}
Let $(M,\omega,J)$ be a Kähler manifold with complex dimension $n$. A Lagrangian submanifold $X \subset M$  is called special Lagrangian if the argument of $\psi$ is constant.
\end{definition}

In \cite{SYZ} Strominger-Yau-Zaslow suggest that mirror Calabi-Yau manifolds can be fibred over the same base in such way that the fibers are special Lagrangian tori. The idea is the following: if $M$ is Calabi-Yau with a special Lagrangian torus fibration $f:M\to B$, the mirror $M^{\vee}$ is the \textit{moduli space} of pairs $(L,\nabla)$ where $L$ is a special Lagrangian torus on $M$ and $\nabla$ is a \textit{flat} connection on $L$. Since the dual torus can be realized as \textit{the moduli space} of \textit{flat} connections (see \cite{Mcl}) mirror manifolds are fibred over the same base $B$ and the fiber over each point is the dual torus.



In the case of elliptic curves, the Lagrangian torus fibration has no singularities and the mirror map can be described as follows:

\begin{example}\label{elip}
Consider the elliptic curve $M=\mathbb{C}/(\mathbb{Z}\oplus \tau\mathbb{Z})$, where for simplicity we assume $\tau=i\gamma\ \mbox{ with }\  \gamma \in \mathbb{R}_+$, holomorphic volume form $\varOmega=dz$ and Kähler form $\omega$ such that $$\displaystyle \int_M\omega=\lambda.$$

Note that the complex and symplectic structures on $M$ are determined by $\gamma$ e $\lambda$ respectively. The mirror elliptic curve has the complex structure determined by $\lambda$ and symplectic structure determined by $\gamma$; that is,   $M^{\vee}=\mathbb{C}/(\mathbb{Z}\oplus i\lambda\mathbb{Z})$ with Kähler form $\omega^{\vee}$ such that $$\displaystyle \int_{M^{\vee}} \omega^{\vee}=\gamma.$$ For more details, see for instance \cite{Aur}.
\end{example}

\subsection{Homological mirror symmetry}

Kontsevich in \cite{Kon} suggests a new approach to mirror symmetry based on an equivalence of categories. Such categories are built in terms of the complex and symplectic geometry of the Calabi-Yau manifold $M$ and its (Calabi-Yau) mirror $M^{\vee}$. This version of mirror symmetry is called homological mirror symmetry (HMS). 

Homological mirror symmetry predicts that the derived category of coherent sheaves on the Calabi-Yau manifold $M$ is equivalent to the Fukaya category on the mirror Calabi-Yau manifold $M^{\vee}$. More precisely:

\begin{conjecture}[Kontsevich]\label{co-hms}
If $(M,J,\omega)$ and $(M^{\vee},J^{\vee},\omega^{\vee})$ are mirror manifolds, then 
$$ D^b Fuk(M,\omega) \cong D^b Coh (M^{\vee},J^{\vee})$$
$$ D^b Coh(M,J) \cong D^b Fuk (M^{\vee},\omega^{\vee})$$ are equivalences of categories.
\end{conjecture}

In Conjecture \ref{co-hms}, $Fuk(M,\omega)$ is the Fukaya category of $(M,\omega)$ whose objects are Lagrangian submanifolds of $(M,\omega)$ equipped with a flat bundle, whose morphims are given by Lagrangian intersection theory, and whose compositions are given by Floer homology.

The homological mirror conjecture remains open for an arbitrary Calabi-Yau manifold. For elliptic curves it was proved by Polishchuk and Zaslow in \cite{PZ}. Denote by $E^{\gamma}_{\lambda}$ the elliptic curve with complex structure given by $\mathbb{C}/\mathbb{Z}\oplus i\gamma\mathbb{Z}$ and symplectic structure given by the $2$-form $\omega$ such that $\int_E\omega=\lambda$. 


\begin{theorem}[Polishchuk-Zaslow]\label{tpz}
Homological mirror symmetry holds for elliptic curves. The mirror of $E^{\gamma}_{\lambda}$ is $E^{\lambda}_{\gamma}$; that is, $$D^bCoh(E^{\gamma}_{\lambda})=D^bFuk(E_{\gamma}^{\lambda}).$$
\end{theorem}

\section{Mirror symmetry via generalized complex geometry} 

We start this section with an remark about $T$-duality of 2-dimensional tori. 

\begin{proposition}\label{toros}
Let $M$ and $\tilde{M}$ be 2-dimensional tori. Then $M$ and $\tilde{M}$ are $T$-dual (in the sense of BHM)
\end{proposition}
\begin{proof}
Consider $M$ and $\tilde{M}$ as $S^1$-bundles over $S^1$. We have $\varOmega ^3(M)=\varOmega ^3(\tilde{M})=\{0\}$. 
In this case we just take the closed $2$-form on $M \times_{S^1} \tilde{M}= T^3$ given by $F=\theta \wedge \tilde{\theta}$, where $\theta$ and $\tilde{\theta}$ are connection forms on $M$ and $\tilde{M}$, respectively.
\end{proof}


Now we will construct the isomorphism $\varphi$ explicitly in the case of two $T$-dual 2-tori $(M,H)$ and $(\tilde{M},\tilde{H})$. Denote by $\partial_{\theta}$ and $\partial_{\tilde{\theta}}$ the duals of the connection forms $\theta$ and $\tilde{\theta}$. As shown in \cite{CG}, we can use $\theta$ and $\tilde{\theta}$ to split the spaces of invariant sections as
$$(\TM)/S^1 \cong TS^1 \oplus \langle \partial_{\theta} \rangle\oplus T^*S^1 \oplus \langle \theta \rangle$$
$$(T\tilde{M}\oplus T^*\tilde{M})/S^1 \cong TS^1 \oplus \langle \partial_{\tilde{\theta}} \rangle \oplus T^*S^1 \oplus \langle \tilde{\theta} \rangle$$ 
Therefore, each element of $(\TM)/S^1$ decomposes into $X+a\partial_{\theta}+\xi+b\theta$, with $a,b \in \mathbb{R}$.


\begin{theorem}
With the notation above, the isomorphism $\varphi:(TM\oplus T^*M)/S^1   \to(T\tilde{M}\oplus T^*\tilde{M})/S^1$ is given by 
\begin{equation}\label{iso-phi}
\varphi(X+a\partial_{\theta}+\xi+b\theta)=X+b\partial_{\tilde{\theta}}+\xi+a\tilde{\theta}.
\end{equation}
\end{theorem}
\begin{proof}
The pull back of $X+a\partial_{\theta}+\xi+b\theta$ is $$X+a\partial_{\theta}+c\partial_{\tilde{\theta}}+\xi+b\theta.$$ Applying the $B$-transform defined by $-F$, we have $$X+a\partial_{\theta}+c\partial_{\tilde{\theta}}+\xi+b\theta+a\tilde{\theta}-c\theta=
X+a\partial_{\theta}+c\partial_{\tilde{\theta}}+\xi+(b-c)\theta+a\tilde{\theta}.$$ For $\xi+(b-c)\theta+a\tilde{\theta}$ to be the pull-back of a form on $\tilde{M}$, we need $b=c$. Finally, we have $$\varphi(X+a\partial_{\theta}+\xi+b\theta)=X+b\partial_{\tilde{\theta}}+\xi+a\tilde{\theta}.$$

\end{proof}

Now we will analyze the behavior of complex and $B$-symplectic (complexified symplectic) structures on elliptic curves under the isomorphism $\varphi$. Recall that both structures are determined by complex numbers with positive imaginary part.

\begin{theorem}\label{teo-princ}
The isomorphism $\varphi$ defined via expression (\ref{iso-phi}) send the complex structure on $M$ defined by $\tau = b+ia$ to a $B$-symplectic structure (completely determined by $\tau$) on the $T$-dual $\tilde{M}$. 
On the other hand, starting with the $B$-symplectic structure on $M$ determined by $\omega_{\mathbb{C}}=b+ia$, $\varphi$ yields on $\tilde{M}$ a complex structure completely determined by $\omega_{\mathbb{C}}$.
\end{theorem}

\begin{proof}
We saw that for T-dual $2$-dimensional tori $M$ e $\tilde{M}$, $\varphi$ just swaps the tangent and cotangent coefficients of the fiber . Considering the decomposition $$(\TM)/S^1\cong TS^1 \oplus \langle\partial_{\theta}\rangle \oplus T^*S^1 \oplus \langle \theta \rangle $$ we can write $\varphi$ as a matrix  $$\varphi=\left(\begin{array}{cccc}
1 & 0 & 0 & 0\\
0 & 0 & 0 & 1\\
0 & 0 & 1 & 0\\
0 & 1 & 0 & 0 \end{array}\right).$$

Given the complex parameter $\tau=b+ia$ on $M$, the correspondent complex structure is 
$$J=\left(\begin{array}{cc}b & \frac{-1-b^2}{a}\\
a & -b \end{array}\right).$$ Regarding it as a generalized complex structure we have $$\J_J=\left(\begin{array}{cc}
 -J & 0\\
  0 & J^* \end{array}\right)=\left(\begin{array}{cccc}
-b & \frac{1+b^2}{a} & 0 & 0\\
-a & b & 0 & 0\\
0 & 0 & b & a\\
0 & 0 & \frac{-1-b^2}{a} & -b\end{array}\right).$$

Now we transfer ${\J}_J$ to $\tilde{M}$ using $\varphi$, as done in Corollary \ref{coro}, obtaining
\begin{eqnarray*}
\tilde{\J}_J&=&\varphi^{-1}\circ \J_J \circ \varphi\\
&=& \left(\begin{array}{cccc}
1 & 0 & 0 & 0\\
0 & 0 & 0 & 1\\
0 & 0 & 1 & 0\\
0 & 1 & 0 & 0 \end{array}\right)\cdot \left(\begin{array}{cccc}
-b & \frac{1+b^2}{a} & 0 & 0\\
-a & b & 0 & 0\\
0 & 0 & b & a\\
0 & 0 & \frac{-1-b^2}{a} & -b\end{array}\right)\cdot\left(\begin{array}{cccc}
1 & 0 & 0 & 0\\
0 & 0 & 0 & 1\\
0 & 0 & 1 & 0\\
0 & 1 & 0 & 0 \end{array}\right)\\
&=&\left(\begin{array}{cccc}
-b & 0 & 0 &\frac{1+b^2}{a}\\
0 & -b & \frac{-1-b^2}{a} & 0\\
0 & a & b & 0\\
-a & 0 & 0 & b \end{array}\right).
\end{eqnarray*}

Then $\tilde{\J}_J=exp(-B) \circ \J_{\omega} \circ exp(B)$ where $$\omega=\left(\begin{array}{cc}
0 & \frac{a}{1+b^2}\\
-\frac{a}{1+b^2} & 0\end{array}\right)\ \ \mbox{and}\ \ B=\left(\begin{array}{cc}
0 & \frac{ba}{1+b^2}\\
-\frac{ba}{1+b^2} & 0\end{array}\right). $$ That is,$\tilde{\J}_J$ is the B-symplectic structure correspondent to the complex number $$\rho=\displaystyle \frac{ba}{1+b^2}+i\left(\frac{a}{1+b^2}\right).$$

Similarly, if we start with complexified symplecitc structure on $M$ determined by $b+ia$, which in our setting corresponds to the symplectic structure $\omega$ and $B$-field given by $$\omega=\left(\begin{array}{cc}
0 & a\\
-a & 0
\end{array}\right),\ \ \ B=\left(\begin{array}{cc}
0 & b\\
-b & 0\end{array}\right) $$ and apply $\varphi$ to the generalized complex structure $exp(-B) \circ \J_{\omega} \circ exp(B)$, on the dual torus we will obtain  the invariant generalized complex structure induced by the invariant complex structure $$\left(\begin{array}{cc}\frac{b}{a} & -\frac{1}{a}\\
a+\frac{b^2}{a} & -\frac{b}{a}\end{array}\right),$$ which corresponds to the complex parameter $$\tau=\displaystyle \frac{b}{a}+i\left(a+\frac{b^2}{a}\right).$$ 
\end{proof}

If we assume a trivial $B$-field on the complexified symplectic form one can recover the mirror map described by  Polishchuk and Zaslow (see Chapter $4$ of \cite{PZ}). Let us remember the notation of Theorem \ref{tpz}: let $\tau=ip$ and $\rho=iq$ be pure imaginary numbers with $p,q>0$. Then  $E^{\tau}_{\rho}$ is the elliptic curve with complex structure determined by $\tau$ and symplectic form determined by $\rho$.


\begin{corollary}
Let $\tau$ and $\rho$ pure imaginary numbers as above. Then the isomorphism $\varphi$ interchange complex and symplectic structures on the elliptic curve $E$, that is, the map $\varphi$ send $E^{\tau}_{\rho}$ to $E^{\rho}_{\tau}$.
\end{corollary}
\begin{proof}
Just set $b=0$ in the proof of Theorem \ref{teo-princ}.
\end{proof}


\end{document}